\documentclass{amsart}
\usepackage{amsfonts,amsmath,amsthm,amssymb}

\newtheorem{theorem}{Theorem}
\newtheorem{corollary}{Corollary}
\newtheorem{lemma}{Lemma}

\newtheorem{remark}{Remark}

\newtheorem{problem}{Problem}

\def\gp#1{\langle #1 \rangle}
\def\m1{^{-1}}
\DeclareMathOperator{\rad}{{\mathfrak J}}

\begin{document}
\title[On filtered multiplicative bases]{On filtered multiplicative bases of some associative algebras}
\author{V.~BOVDI, A.~GRISHKOV, S.~SICILIANO}

\address{
\texttt{VICTOR BOVDI, ALEXANDER GRISHKOV}
\newline
IME, USP, Rua do Matao, 1010 -- Citade Universit\`{a}ria, CEP
05508-090, Sao Paulo, Brazil}
\email{\{vbovdi, shuragri\}@gmail.com}

\address{
\texttt{SALVATORE SICILIANO},
\newline
Dipartimento di Matematica e Fisica ``Ennio De Giorgi",
Universit\`{a} del Salento,
Via Provinciale Lecce--Arnesano, 73100--LECCE, Italy}
\email{salvatore.siciliano@unisalento.it}
\thanks{The research was supported by FAPESP (Brazil) Process N: 2014/18318-7,
CNPq(Brazil) and RFFI 13-01-00239a(Russia).}

\maketitle

\begin{abstract}
We deal with the existing problem of filtered multiplicative bases of finite-dimensional associative algebras. For an associative algebra $A$ over a field, we investigate when the property of having a filtered multiplicative basis is hereditated by  homomorphic images or by the associated graded algebra of $A$. These results are then applied to some classes of  group algebras and restricted enveloping algebras.
\end{abstract}

\section{Introduction}

Let $A$ be an associative algebra over a field
$F$ and denote by $\rad(A)$ the Jacobson radical of $A$. An $F$-basis $\mathfrak{B}$ of $A$ is called \emph{multiplicative} if $\mathfrak{B}\cup \{ 0\}$ is a semigroup under the product of $A$. If one also has that  $\mathfrak{B}\cap\rad(A)$ is an $F$-basis of $\rad(A)$, then $\mathfrak{B}$ is said to be a \emph{filtered multiplicative basis} (shortly, f.m.b.) of $A$.
Filtered multiplicative bases   arise in the representation theory
 of associative algebras and were introduced by
H.~Kupisch in \cite{Kupisch}.

In \cite{BGRS},  R.~Bautista, P.~Gabriel, A.~Roiter and
L.~Salmeron proved that if a finite-dimensional associative algebra $A$ has  finite representation type
over an algebraically closed field $F$, then $A$ has an f.m.b. This implies that the number of isomorphism classes of  algebras of finite representation type of a given dimension is finite and reduces the classification of these algebras to a combinatorial problem. In the same paper \cite{BGRS} it was asked when a group algebra has an f.m.b.
and such a problem (not necessarily for group algebras) has been subsequently considered by
several authors: see e.g. \cite{Balogh_1, Balogh_2, Bovdi_1,
Bovdi_2, Calderon, Landrock_Michler, Roiter_Sergeichuk, Zhu_Li}. In particular, it is still an open problem whether
a group algebra $F G$  has  an f.m.b. in the case when  $F$
is a field  of odd characteristic $p$ and $G$ is a nonabelian
$p$-group (see \cite{Oberwolfach}, Question 5).

Moreover, in \cite{Bovdi_Grishkov_Siciliano} the same problem was investigated in the setting of
restricted enveloping algebras $u(L)$, where $L$ is in the class $\mathfrak{F}_p$ of finite-dimensional
and $p$-nilpotent restricted Lie algebras over a field of positive characteristic $p$. In particular, we
characterized commutative restricted enveloping algebra having an  f.m.b., and showed that if $L$ has
nilpotency class 2 and $p>2$ then $u(L)$ does not have any f.m.b.

The aim of the present paper  is to provide some further contribution on the problem of existence of an f.m.b. for an associative algebra.
 First, we deal with the conditions under which the property of having a multiplicative basis is inherited by homomorphic images. This result is then used to establish when a restricted enveloping algebra $u(L)$  has an f.m.b., where  $L\in \mathfrak{F}_2$ has nilpotency class 2 over a field of characteristic 2, thereby complementing the previous results in \cite{Bovdi_Grishkov_Siciliano}. Next, we show that if a finite-dimensional associative algebra $A$ admits an f.m.b., then so does its graded algebra associated  to the filtration given by the powers of the  Jacobson radical. The combination of such a result with \cite{Bovdi_Grishkov_Siciliano}  allows to conclude that if $F$ is a
field  of odd characteristic $p$ and $G$ is a finite $p$-group of nilpotency class 2, then the group algebra $F G$  has no  f.m.b.,
which provides a partial answer to the question 5 in \cite{Oberwolfach}.

In the sequel we will use freely the notation and results from the books \cite{Bovdi_A, Strade_Farnsteiner_book}.

\section{Preliminaries}

Let $A$ be a finite-dimensional associative algebra over a field
$F$ having an f.m.b.  $\mathfrak{bs}(A)$. Then the following
simple properties hold (see \cite{Bovdi_1}):
\begin{itemize}
\item[(F1)] $\mathfrak{bs}(A)\cap \rad^n(A)$ is an $F$-basis of
$\rad^n(A)$ for every $n\geq 1$;
\item[(F2)] if $u,v\in \mathfrak{bs}(A) \backslash \rad^k(A)$ and
$u\equiv v\pmod{\rad^k(A)}$
then $u=v$;
\item[(F3)] if another $F$-algebra $B$ admits an f.m.b. then so does $A\otimes_{F} B$.
\end{itemize}
We denote by $A^-$  the restricted Lie algebra associated to $A$ via the Lie bracket $[x,y]=xy-yx$ for every $x,y\in A$ and $p$-map given by ordinary $p$-exponentiation. For a subset $S$ of $A$ we denote by $\langle S \rangle$ and $\langle S\rangle_{ F}$, respectively, the associative subalgebra and the $ F$-vector subspace spanned by $S$.

Let $L$ be a restricted Lie algebra over a field $F$ of positive
characteristic $p$ with a $p$-map $[p]$. We denote by $\omega(L)$ the
\emph{augmentation ideal} of $u(L)$, that is, the associative
ideal generated by $L$ in $u(L)$.  The restricted ideals of $L$ given by
\[
\mathfrak{D}_m(L)=L\cap \omega^m(L), \qquad \quad (m\geq 1)
\]
are called the \emph{dimension subalgebras} of $L$ (see \cite{Riley_Shalev}).
Similarly to the dimension subgroups (in the context of modular group algebras),  these subalgebras can be  explicitly described as
$\mathfrak{D}_m(L)=\sum_{ip^{j}\geq m}{{\gamma}_{i}(L)}^{[p]^{j}}$,
where ${\gamma}_{i}(L)^{[p]^j}$ is the restricted subalgebra of
$L$ generated by the set of $p^j$th powers of the $i$th term
of the lower central series of $L$.
The center of $L$ will be denoted by $Z(L)$.  For a subset $S$ of $L$ we will denote by $\langle S \rangle_p$
the restricted subalgebra generated by $S$.  A restricted Lie algebra $H$ is said to be \emph{nilcyclic} if $H=\langle x \rangle_p$ for some  $p$-nilpotent element $x$ of $H$.

It is well-known that if $L$ is finite-dimensional and
$p$-nilpotent then $\omega(L)$ is nilpotent (see \cite{Strade_Farnsteiner_book}, Corollary 3.7 of Chapter 1).
Clearly, in this case $\omega(L)$  coincides with  $\rad(u(L))$ and
$u(L)=F \cdot 1\oplus \omega(L)$, so that $u(L)$ is a local basic $F$-algebra. In this case, if $u(L)$ has
an  f.m.b. $\mathfrak{bs}(u(L))$, then we can assume without loss of
generality that $1\in \mathfrak{bs}(u(L))$. For each $x\in L$, the
largest subscript $m$ such that $x \in \mathfrak{D}_m(L)$ is
called the \emph{height} of $x$ and denoted by $\nu(x)$. The
combination of Theorem $2.1$ and Theorem $2.3$ from
\cite{Riley_Shalev} yields the following.

\begin{lemma}\label{L:1}
Let $L\in \mathfrak{F}_p$ be a restricted Lie algebra over a field $F$,
and let $\{x_i \}_{i\in I}$ be an ordered basis  of $L$ chosen such  that
\[
\mathfrak{D}_m(L)= \langle x_i\;  \vert\;  \nu(x_i)\geq m \rangle_F\qquad  (m\geq 1).
\]
Then for each positive integer $n$ the following statements hold:
\begin{itemize}
\item [(i)] $\omega(L)^n=\langle x\; \vert \;\nu(x)\geq n \rangle_F$, where\quad  $x=x^{\alpha_1}_{i_1}\cdots
x^{\alpha_l}_{i_l}$,\newline
$\nu(x)=\sum_{j=1}^l\alpha_j\nu(x_{i_j})$, \quad $i_1<\cdots
<i_l$\quad and\quad  $0\leq \alpha_j \leq p-1$. \item [(ii)] The
set $\{\; y \; \vert\;  \nu(y)=n \;\}$ is an $F$-basis of
$\omega(L)^n$ modulo $\omega(L)^{n+1}$.
\end{itemize}
\end{lemma}
If $S$ is a subset of a $p$-nilpotent restricted Lie algebra then  the minimal positive
integer $n$ such that $z^{[p]^n}=0$ for every $z\in S$ is called the
\emph{exponent} of $S$ and denoted by $e(S)$.

\section{Homomorphic images and restricted enveloping algebras}

Let $A$ be an associative algebra over a field $F$ and let $\mathfrak{bs}(A)$ be an $F$-basis of $A$. A subset $P\subset A$ is called  $\mathfrak{bs}(A)$-\emph{regular} if for every $x\in P$ one has that either $x\in  \mathfrak{bs}(A)$ or $x=a-b$ for some $a,b\in  \mathfrak{bs}(A)$.

\begin{theorem}\label{T:1}
Let $A=\gp{g_1,\dots,g_m}$  be a finitely generated   associative algebra over a field $F$ and let $\mathfrak{bs}(A)=\{a_i\mid i\in I\}$ be a multiplicative basis of $A$ such that  $\{g_1,\dots,g_m\}\subseteq \mathfrak{bs}(A)$. Let $\psi:A\to B$ be a surjective homomorphism of $A$ onto an $F$-algebra $B$ and $H=\{ \psi(g_i)\vert  i=1,\ldots,m\}$. Then the following statements hold:
\begin{itemize}
\item [{\normalfont (i)}] If there exists $J\subseteq I$ such that $\mathfrak{bs}(B)=\{b_i=\psi(a_i)\vert i\in J\}$ is a multiplicative basis of $B$ containing $H$, then $\mathfrak{Ker}(\psi)$ has a  $\mathfrak{bs}(A)$-regular $F$-basis.
\item [{\normalfont (ii)}] If $\mathfrak{Ker}(\psi)$ has a  $\mathfrak{bs}(A)$-regular $F$-basis, then  there exists $J\subseteq I$ such that $\mathfrak{bs}(B)=\{b_i=\psi(a_i)\vert i\in J\}$ is a multiplicative basis of $B$.
\end{itemize}
\end{theorem}

\begin{proof}
(i)     For every $i,j\in I$ we have either    $b_ib_j=0$ or $b_ib_j=b_k$ for some  $k\in J$.
Denote by $K$ the $F$-vector space spanned   by the set
\[
Z=\{\; a_i\in \mathfrak{bs}(A)\;   \vert \;  \psi(a_i)=0\;  \}\cup
\{\;   a_j-a_k \;  \vert  \; \psi(a_j)=\psi(a_k), \;  a_j,a_k\in \mathfrak{bs}(A)\; \}.
\]
Clearly  $K\subseteq\mathfrak{Ker}(\psi)$.
Let us prove that $K=\mathfrak{Ker}(\psi)$.

Let    $v=\sum_{i\in X}\alpha_i a_i\in\mathfrak{Ker}(\psi)\setminus K$,  such that $\alpha_i\neq 0$  for all $i$ in the finite subset $X\subseteq I$. Let us choose the element $v$ such that  the cardinality $\vert X\setminus J\vert$  is minimal. If $X\subseteq J$ then $\{b_i=\psi(a_i) \vert \, i\in X\}$ is an  $F$-linear dependent subset of the algebra $B$, a contradiction. Hence there exists  $i\in X\setminus J$ and $a_i=w(g_1,\dots,g_m)$,
where  $w(x_1,\dots,x_m)$ is a monomial in the free associative algebra $F\gp{x_1,\dots,x_m}$.
It follows that
\[
b_i=\psi(a_i)=w\big(\psi(g_1),\dots,\psi(g_m)\big).
\]
Since  $\mathfrak{bs}(B)=\{b_i\vert i\in J\}$ is a multiplicative basis  and $\psi(g_1),\dots,\psi(g_m)\in \mathfrak{bs}(B)$, we get
$w(\psi(g_1),\dots,\psi(g_m))=b_j$ for some $j\in J$. Therefore $\psi(a_i)=\psi(a_j)$ and $a_i-a_j\in K$. Fix the natural numbers $i,j$ and  put $X_0=(X\setminus \{i\})\cup \{j\}$. Then
\[
v+\alpha_i(a_j-a_i)=\sum_{s\in X_0}\alpha_s a_s\in \mathfrak{Ker}(\psi)\setminus K
\]
where $\alpha_j=\alpha_i$ and $\vert X_0\setminus J \vert=\vert X\setminus J \vert -1$, a contradiction. Hence $K=\mathfrak{Ker}(\psi)$ and,  in particular,  $\mathfrak{Ker}(\psi)$ has a basis which is a  $\mathfrak{bs}(A)$-regular set.

(ii)  Assume that  $\mathfrak{Ker}(\psi)$ has a basis ${\bf K}$ which is a $\mathfrak{bs}(A)$-regular set. By the Zorn's Lemma  we can assume that $I$ is a well-ordered set.

Put $I_0=\{i\in I\vert  a_i\in\mathfrak{Ker}(\psi)\}$ and
\[
I_1=\{\quad  i\in I \setminus I_0\quad  \vert \quad a_i-a_j\in\mathfrak{Ker}(\psi) \text{ for some}\, j>i \quad  \}.
\]
Define  the function  $\mathfrak{p}:I_1\to I$ by
\[
I_1\ni i\mapsto \min\{\quad  j\in I\setminus I_0\quad \vert \quad  j>i, \; a_i-a_{j}\in\mathfrak{Ker}(\psi)\quad \}.
\]
Put $I_2=I_1\setminus \mathfrak{p}(I_1)$.

 We split the proof in several steps:

\emph{Step 1}: If $i,j\in I_1$ and $i<j$ then $\mathfrak{p}(i)\not=\mathfrak{p}(j)$.

Let  $\mathfrak{p}(i)=\mathfrak{p}(j)$. Clearly $
a_i-a_{\mathfrak{p}(i)}, a_j-a_{\mathfrak{p}(j)}\in\mathfrak{Ker}(\psi)$
and
\[
a_i-a_j=(a_i-a_{\mathfrak{p}(i)})-( a_j-a_{\mathfrak{p}(j)})\in\mathfrak{Ker}(\psi),
\]
so that $\mathfrak{p}(i)\leq j<\mathfrak{p}(j)=\mathfrak{p}(i)$, a contradiction.

\emph{Step 2}:  For every $i\in I_1$ we define  the corresponding $i$-ray
\[
R(i)=\{\mathfrak{p}_0(i)<\mathfrak{p}_1(i)<\mathfrak{p}_2(i)<\cdots\mid \mathfrak{p}_k(i)\in I\},
\]
where $\mathfrak{p}_0(i)=i$, $\mathfrak{p}_1(i)=\mathfrak{p}(i)$  and, moreover,  $\mathfrak{p}_{n+1}(i)=\mathfrak{p}(\mathfrak{p}_n(i))$
if $\mathfrak{p}_n(i)\in I_1$ while $\mathfrak{p}_{n+1}(i)$ is not defined if $\mathfrak{p}_n(i)\notin I_1$.

Note that every $i$-ray is contained in a unique maximal $j-$ray. Moreover, a $j$-ray is maximal for $j\in I_1$ if and only if $j\in I_2$.
The former part follows from the fact that the well-ordered set $I$ does not contain any infinite decreasing chain, and the latter one is trivial as for every $j\in \mathfrak{p}(I_1)$ with $\mathfrak{p}(i)=j$ we have $R(j)\subset R(i)$.

\emph{Step 3}: For every two different maximal rays $R(i)$ and $R(j)$ we have $R(i)\cap R(j)=\emptyset$
and for every $i\in I_1$ there exists a unique minimal $f(i)\in I_2$ such that $i\in R(f(i))$.
Moreover, for $i,j\in I_1$ we have $\psi(a_i)=\psi(a_j)$ if and only if
$f(i)=f(j).$ In particular, as $f(f(i))=f(i)$, for every $i\in I_1$ we have that  $a_i-a_{f(i)}\in\mathfrak{Ker}(\psi)$.

Let us prove that the $\mathfrak{bs}(A)$-regular set
\[
 {\bf K}_1=\{a_i|i\in I_0\}\cup\{a_i-a_{\mathfrak{p}(i)}|i\in I_1\}
\]
is a basis of $\mathfrak{Ker}(\psi)$.

\emph{ Step 4}: $\langle {\bf K}_1\rangle_{F}=\mathfrak{Ker}(\psi)$.
As $\mathfrak{Ker}(\psi)$ has a $\mathfrak{bs}(A)$-regular basis ${\bf K}$ it is enough to prove that if $a_i-a_j\in\mathfrak{Ker}(\psi)$ with $j>i\not\in I_0$ then $a_i-a_j\in \langle {\bf K}_1\rangle_{F}$. By Step 3 we have $f(i)=f(j)=k$ and so
\[
i,j\in R(k)=\{k<k_1<k_2<\cdots<k_s=i<\cdots<j=k_{s+t}<\cdots\}.
\]
It follows that $a_i-a_j=\sum_{l=s}^{s+t-1}(a_{k_l}-a_{k_{l+1}})$ and $a_{k_l}-a_{k_{l+1}}\in {\bf K}_1$ for every $s\leq l <s+t$, yielding the claim.

 \emph{Step 5}: The set ${\bf K}_1$ is $F$-linearly independent.

Let $\sum_{i\in I_0}\alpha_ia_i+\sum_{j\in I_1}\beta_j(a_j-a_{\mathfrak{p}(j)})=0$,  where $\alpha_i,\beta_j\in F$ for every $i \in I_0$ and $j\in I_1$. As $I_0\cap I_1=\emptyset$ we have $\alpha_j=0$ for all $j\in I_0$.

Suppose  that $\beta_s \neq 0$ for some $s\in I_1$. Put $j=max\{s\vert \, \beta_s\neq 0\}$. Then  $\mathfrak{p}(j)>i$ for every $i$ such that $\beta_i\not=0$. It follows that $\beta_j=0$, a contradiction.

 \emph{Step 6.}  $\mathfrak{B}=\{b_i=\psi(a_i)|i\in I_3= I\setminus (I_0\cup \mathfrak{p}(I_1))\}$ is a multiplicative basis of $B$.

 Observe that if $b_i=\psi(a_i)\in B$ and $i\not\in I_3$, then either
 $i\in I_0$ and $b_i=0$ or $i\in \mathfrak{p}(I_1)$ and so $\psi(a_i)=b_i=\psi(a_{f(i)})=b_{f(i)}\in \mathfrak{B}$ with $f(i)\not\in I_0\cup \mathfrak{p}(I_1)$. As a consequence, we have $\langle \mathfrak{B}\rangle_{F}=B$.

Suppose now that $\sum_{i\in I_3}\beta_ib_i=0$  for some $\beta_i\in F$. Then $\sum_{i\in I_3}\beta_ia_i\in\mathfrak{Ker}(\psi)$ and so,  by Steps 4 and 5, we get
\begin{equation}\label{E:1}
\sum_{i\in I_3}\beta_ia_i=\sum_{j\in I_0}\alpha_ja_j+\sum_{s\in I_1}\gamma_s(a_s-a_{\mathfrak{p}(s)}).
\end{equation}
Since $I_0\cap (I_1\cup I_3)=\emptyset$ we have  $\alpha_j=0$ for every $i\in I_0$.  Let $t=\max\{s\vert \,\gamma_s\not=0\}$.
As $j<\mathfrak{p}(j)\leq \mathfrak{p}(t)$ for every $j$ such that $\beta_j \neq 0$ and  $\mathfrak{p}(t)\not\in I_3$,  relation (\ref{E:1})  forces  $\gamma_t=0$, a contradiction. Thus $\mathfrak{B}$ is an $F$-basis of $B$. It remains to show that $\mathfrak{B}$ is multiplicative. Let $i,j\in I_3$. Then there exists $k\in I$ such that $a_ia_j=a_k$. If $k\in I_0$ then one has $b_ib_j=0$. On the other hand, if $k\in I_3$ then $b_ib_j\in \mathfrak{B}$. Finally, if $k\not\in \left( I_0\cup I_3\right)$ then $k\in \mathfrak{p}(I_1)$, so that $b_k=b_{f(k)}\in \mathfrak{B}$.
\end{proof}

\begin{remark}
Suppose that $A$ is a finitely generated  associative algebra over a field $F$ having a multiplicative basis $\mathfrak{bs}(A)$. Assume that we can choose a  minimal set of generators $\{a_1,\ldots,a_n\}$ of $A$ such that $\{a_1,\ldots,a_n\}\subseteq \mathfrak{bs}(A)$. For a set $X=\{x_1,\ldots,x_n\}$ we denote by $F\langle X \rangle$ the free $F$-associative algebra over $X$ and by $X^\ast$ the free monoid on  $x_1,\ldots,x_n$. Clearly, there exists an homomorphism $\psi:F\langle X \rangle \to A$ such that
$\psi(x_i)=a_i$ and $\mathfrak{Ker}(\psi)$ has an $X^\ast$-regular $F$-basis (by Theorem  \ref{T:1}).
\end{remark}

Let $\mathfrak{L}$ be the relatively free nilpotent restricted Lie algebra of class 2 on the set $\{x,y\}$ over a field of characteristic 2. Denote by $\mathfrak{c}_{(s)}$ the nilcyclic restricted Lie algebras of exponent $s$ and put $\mathfrak{h}_{(s)}=\mathfrak{L}/I$, where $I$ is the restricted ideal of $\mathfrak{L}$ generated by $x^{[2]^s},y^{[2]^s}$ and $[x,y]^{[2]^s}$. For every $m,n\geq 0$ and $s>0$ in the sequel we will use the restricted Lie algebra
\[
L(m,n;s)=\underbrace{\mathfrak{c}_{(s)} \oplus \cdots \oplus \mathfrak{c}_{(s)}}_{m}\oplus \underbrace{\mathfrak{h}_{(s)}\oplus  \cdots \oplus \mathfrak{h}_{(s)}}_{n}.
\]
The restricted enveloping algebra of $L(m,n;s)$ admits an  f.m.b. Indeed we have

\begin{lemma}\label{L:2}
For every $m,n\geq 0$ and $s>0$ the associative algebra $u(L(m,n;s))$ has a filtered multiplicative basis.
\end{lemma}
\begin{proof} Since we have
\[
u(L(m,n;s))\cong \underbrace{u(\mathfrak{c}_{(s)}) \otimes_{F} \cdots \otimes_{F}u(\mathfrak{c}_{(s)})}_{m}\otimes_{F} \underbrace{u(\mathfrak{h}_{(s)}) \otimes_{F} \cdots  \otimes_{F}u(\mathfrak{h}_{(s)})}_{n},
\]
by virtue of (F3) and Theorem 1 of \cite{Bovdi_Grishkov_Siciliano} it enough to show that $u(\mathfrak{h}_{(s)})$ has an f.m.b.
Let $\mathfrak{L}$ be the  relatively free nilpotent restricted Lie algebra of class 2 on the set $\{x,y\}$ and $I$ the restricted ideal of $\mathfrak{L}$ generated by $x^{[2]^s},y^{[2]^s}$ and $[x,y]^{[2]^s}$. Consider the unique associative homomorphism $\hat{\pi}:u(\mathfrak{L})\to u(\mathfrak{h}_{(s)})$ extending the canonical map
$\pi: \mathfrak{L} \to \mathfrak{L}/I=\mathfrak{h}_{(s)}$. As $\mathfrak{L}^{[2]}\subseteq Z(\mathfrak{L})$, it is clear that $\mathfrak{Ker}(\hat{\pi})=Iu(L)$ is spanned by the elements of the form $x^{2^s}\omega_1$, $y^{2^s}\omega_2$, $\left((xy)^{2^s}+(yx)^{2^s}\right)\omega_3$, where the $\omega_i$ are monomials in $x,y$. Consequently, by Theorem \ref{T:1}
we see that $u(\mathfrak{h}_{(s)})$ has a multiplicative basis $\mathfrak{bs}(u(\mathfrak{h}_{(s)}))$ with $\mathfrak{bs}(u(\mathfrak{h}_{(s)})\backslash \{ 1\} \subset \omega(\mathfrak{h}_{(s)})$. Finally, as $\mathfrak{h}_{(s)}$ is finite-dimensional and $p$-nilpotent we have $\omega(\mathfrak{h}_{(s)})=\rad(u(\mathfrak{L}))$, so that $\mathfrak{bs}(u(\mathfrak{h}_{(s)}))$ contains an $F$-basis of $\rad(u(\mathfrak{h}_{(s)}))$.\end{proof}

We say that an associative algebra $A$ over a field of characteristic 2 is of \emph{Heisenberg type} if there exist $m,n\geq 0$, $s>0$, and an f.m.b. $\mathfrak{B}$ of $u(L(n,m,s))$ such that
$A\cong u(L(m,n;s))/J$ for some ideal $J$ of $u(L(m,n;s))$ having a $\mathfrak{B}$-regular basis.

Let $L$ be a finite-dimensional $p$-nilpotent restricted Lie algebra over a field of characteristic $p>0$. In \cite{Bovdi_Grishkov_Siciliano} we proved that if $L$ is abelian then $u(L)$ has a filtered multiplicative basis if and only if it is a direct sum of cyclic restricted subalgebras. Moreover, we showed that if $L$ has nilpotent class 2 and $p>2$ then $u(L)$ does not have any filtered multiplicative basis. Here we prove the following:

\begin{theorem}\label{T:2}
If $L\in {\mathfrak F}_p$ has nilpotency class 2 then  $u(L)$ has a filtered multiplicative basis if and only if  $p=2$
and $u(L)$ is of  Heisenberg type.
\end{theorem}

\begin{proof}
Suppose that  $u(L)$ has an f.m.b. $\mathfrak{bs}(u(L))$ such that  $1\in \mathfrak{bs}(u(L))$. By Theorem 3 of \cite{Bovdi_Grishkov_Siciliano} the ground field must have characteristic 2. Let
\[
{\bf \Gamma}=\mathfrak{bs}(u(L)) \backslash \left(\omega(L)^2\cup\{1\}\right)=\{g_1,\dots,g_t\}.
\]
Then ${\bf \Gamma}$ is a minimal set of generators of $u(L)$ as a unitary associative $F$-algebra and, moreover,  by property (F1) and Lemma \ref{L:1}, for every $i=1,\ldots,t$ there exists $c_i\in L$ such that $c_i\equiv g_i \pmod{\omega^2(L)}$.
As $L$ is not commutative, by Lemma 2 of \cite{Bovdi_Grishkov_Siciliano} there exist $1\leq i,j \leq t$ such that $[c_i,c_j]\not\in \mathfrak{D}_3(L)$.  If  $1\leq k \leq t$ with $k\neq i,j$, as an easy consequence of Lemma \ref{L:1} we deduce the following Facts:
 \begin{itemize}
 \item[{\normalfont(a)}] if $c_i[c_j,c_k]\in \omega(L)^4$ then $[c_j,c_k]\in \mathfrak{D}_3(L)$;
\item[{\normalfont(b)}] if $c_j[c_i,c_k]\in \omega(L)^4$ then $[c_i,c_k]\in \mathfrak{D}_3(L)$.
\item[{\normalfont(c)}]  $c_k[c_i,c_j]+c_j[c_i,c_k]\not\in \omega(L)^4$;
\item[{\normalfont(d)}]  $c_k[c_i,c_j]+c_i[c_j,c_k]\not\in \omega(L)^4$;
\item[{\normalfont(e)}] $c_i[c_j,c_k]+c_j[c_i,c_k]+c_k[c_i,c_j]\not\in \omega(L)^4$.
\end{itemize}
Consider the following six elements:
\begin{align*}
\mathfrak{m}_1=g_ig_jg_k&\equiv c_ic_jc_k \pmod {\omega(L)^4};\\
\mathfrak{m}_2=g_ig_kg_j&\equiv  c_ic_jc_k+c_i[c_j,c_k]\pmod {\omega(L)^4};\\
\mathfrak{m}_3=g_jg_ig_k&\equiv  c_ic_jc_k+c_k[c_i,c_j]\pmod {\omega(L)^4};\\
\mathfrak{m}_4=g_jg_kg_i&\equiv  c_ic_jc_k+c_k[c_i,c_j]+c_j[c_i,c_k]\pmod {\omega(L)^4};\\
\mathfrak{m}_5=g_kg_ig_j&\equiv  c_ic_jc_k+c_i[c_j,c_k]+c_j[c_i,c_k]\pmod {\omega(L)^4};\\
\mathfrak{m}_6=g_kg_jg_i&\equiv  c_ic_jc_k+c_i[c_j,c_k]+c_j[c_i,c_k]+c_k[c_i,c_j]\pmod {\omega(L)^4}.\\
\end{align*}
Consequently, by property (F2) we get
\[
\begin{split}
\dim_{F}\Big(\langle\mathfrak{m}_1,\ldots,\mathfrak{m}_6 \rangle_{F}\Big)=&\dim_{F}\Big(\langle\mathfrak{m}_1,\ldots,\mathfrak{m}_6 \rangle_{F}+\omega^4(L)/\omega(L)^4\Big) \leq 4,
\end{split}
\]
 so that we must have $\mathfrak{m}_s=\mathfrak{m}_t$ for some $s\not=t$.
By Facts (a) and (b) we immediately have that
\[
\{\mathfrak{m}_1,\mathfrak{m}_2,\mathfrak{m}_5\} \cap \{\mathfrak{m}_3,\mathfrak{m}_4,\mathfrak{m}_6\}=\emptyset.
\]
We claim  that $[c_i,c_k]\in \mathfrak{D}_3(L)$. Suppose the contrary. Notice that, by Fact (b), we have $\mathfrak{m}_2\not=\mathfrak{m}_5$ and $\mathfrak{m}_3\not=\mathfrak{m}_4$. Now we distinguish two cases:

\emph{Case 1:} $[c_j,c_k]\in \mathfrak{D}_3(L)$.
Then property (F2) yields $\mathfrak{m}_1=\mathfrak{m}_2$ and $\mathfrak{m}_4=\mathfrak{m}_6$ and, moreover, by Lemma \ref{L:1} we have $c_i[c_j,c_k]\in \omega(L)^4$. It follows that
\[
\dim_{F} \langle\mathfrak{m}_1,\ldots,\mathfrak{m}_6 \rangle_{F}<4
\]
and so $\mathfrak{m}_1=\mathfrak{m}_5$ or $\mathfrak{m}_3=\mathfrak{m}_6.$
In both cases we conclude that $c_j[c_i,c_k]\in \omega(L)^4$, contradicting Fact (b).

\emph{Case 2:} $[c_j,c_k]\not\in \mathfrak{D}_3(L)$. Then $c_i[c_j,c_k]\not\in \omega(L)^3$, so that $\mathfrak{m}_1\neq \mathfrak{m}_2$ and $\mathfrak{m}_4\neq\mathfrak{m}_6$. It follows that $\mathfrak{m}_1=\mathfrak{m}_5$ and $\mathfrak{m}_3=\mathfrak{m}_6$ and, in turn,
\[
c_i[c_j,c_k]+ c_j[c_i,c_k]\in \omega(L)^4,
\]
which is impossible by Lemma \ref{L:1}.

Therefore  $[c_i,c_k]\in \mathfrak{D}_3(L)$ and in a similar way one can show that $[c_j,c_k]\in \mathfrak{D}_3(L)$, as well. It follows that $g_ig_k\equiv g_kg_i\pmod{\omega(L)^3}$ and $g_jg_k\equiv g_kg_j\pmod{\omega(L)^3}$ and then, by Lemma \ref{L:1} and (F2),  $g_k$ commutes both with $g_i$ and $g_j$. Thus, for any $g_i\in {\bf \Gamma}$ one has that either $g_i$ is in the center $Z(u(L))$ of $u(L)$ or there exists a unique $g_j\in {\bf \Gamma}$ which does not commute with $g_i$.
We can then reindex the elements of  $ {\bf \Gamma}\backslash \{1\}$ in such a way that  $[g_{2i-1},g_{2i}]\neq 0$ for $i=1,\ldots,r$ and all the other commutators are zero. Consider the restricted Lie algebra $L(m,n;s)$, where $n=t-2m$ and $s$ is the exponent of ${\bf \Gamma}$ in $u(L)^-$.  For every $i=1,\ldots,m$ let $x_i,y_i$ be  generators of $i$th copy of $\mathfrak{h}_{(s)}$ and for every $j=1,\ldots,n$ let $z_j$ be a generator of the $j$th copy of $\mathfrak{c}_{(s)}$. For every $i=1,\ldots,m$ one has
\[
g_{2i-1}^2g_{2i}\equiv c_{2i-1}^2c_{2i} \equiv c_{2i} c_{2i-1}^2 \equiv g_{2i}g_{2i-1}^2\pmod{\omega(L)^4}
\]
and so property (F2) forces $[g_{2i-1}^2,g_{2i}]=0$. Thus $g_{2i-1}^2 \in Z(u(L))$ and in a  similar way one can prove that $g_{2i}^2 \in Z(u(L))$, as well. It follows that $[g_{2i-1},g_{2i}]\in Z(u(L))$ for every $i=1,\ldots,m$. The just proved properties assure the existence a unique surjective restricted homomorphism $\phi: L(m,n;s) \longrightarrow u(L)^-$ such that $\phi(x_i)=g_{2i-1}$ and $\phi(y_i)=g_{2i}$ for every $i=1,\ldots,m$ and $\phi(z_j)=g_{2m+j}$ for every $j=1,\ldots,n$. Let $\tilde{\phi}: u(L(m,n;s)) \longrightarrow u(L)$ denote the unique algebra homomorphism extending $\tilde{\phi}$. Then Theorem  \ref{T:1} allows to conclude that  $A\cong u(L(m,n;s))/J$, where $J=\mathfrak{Ker}(\tilde{\phi})$ is an ideal of $u(L)$ having a $\mathfrak{B}$-regular $F$-basis, proving the necessity part.

The sufficiency part is an immediate consequence of Lemma \ref{L:2} and Theorem  \ref{T:1}. \end{proof}

We conclude this section with  some open problems.
We say that an associative algebra $A$ over a field of characteristic 2 is of \emph{strong Heisenberg type} if there exist $m,n\geq 0$, $s>0$, and an f.m.b. $\mathfrak{B}$ of $u(L(n,m;s))$ such that
$A\cong u(L(m,n;s)/J)$ for some ideal $J$ of $L(m,n;s)$ having a $\mathfrak{B}$-regular basis.
As $u(L(m,n;s)/J)\cong u(L(m,n;s))/Ju(L(m,n;s))$, it is clear that in such a case $A$ is of Heisenberg type.

If the following problem has a positive answer then the conclusion of Theorem \ref{T:2} would be considerably improved:
\begin{problem}\label{heis}
Let $L\in {\mathfrak F}_p$ of nilpotency class 2 over a field of characteristic 2 and suppose that $u(L)$ has an f.m.b. Is $u(L)$ of strong Heisenberg type?
 \end{problem}

Likely, the characterization of the restricted Lie algebras $L$ such that $u(L)$ is of Heisenberg type could be a delicate task involving the isomorphism problem for restricted Lie algebras:

 \begin{problem}\label{fin}
Characterize the restricted Lie algebras $L$ whose restricted enveloping algebra $u(L)$ is of Heisenberg type.
 \end{problem}

Finally, we suspect that the following problem could have a positive answer:

\begin{problem}\label{prob3}
Suppose that $L\in \mathfrak{F}_p$ is not abelian and $u(L)$ has an f.m.b. Is it true that $p=2$ and $L$ is nilpotent of class 2?
 \end{problem}

\section{Associated graded algebras and group algebras}

For an associative algebra $A$, in this section we will consider the associated graded algebra
\[
\mathrm {gr}(A)=\bigoplus_{i\geq 0} \rad^i(A)/ \rad^{i+1}(A),
\]
associated to the filtration given by the powers of the  Jacobson radical $\rad(A)$ of $A$.

\begin{theorem}\label{T:3}
Let $A$ be a finite-dimensional associative algebra over a field $F$.
If $A$ has an f.m.b.,  then $\mathrm {gr}(A)$ has an f.m.b.
\end{theorem}
\begin{proof}
Let $\mathfrak{bs}(A)$ be an f.m.b. of $A$.
Put
\[
\mathfrak{bs}(A)_i=\Big(\mathfrak{bs}(A) \cap  \rad^i(A) \Big) \backslash \rad^{i+1}(A), \qquad (i=0,1,\ldots,n-1)
\]
where $n$ is the nilpotency class of $\rad(A)$.
Then, in view of \cite{Bovdi_1}, the images of the elements of $\mathfrak{bs}(A)_i$ in $A/ \rad^{i+1}(A)$ form an $F$-basis $\overline{\mathfrak{bs}(A)}_i$ for the vector space  $\rad^i(A)/ \rad^{i+1}(A)$, where $i=0,\ldots,n-1$. As a consequences, the set
\[
\overline{\mathfrak{bs}(A)}:= \bigcup_{i=0}^{n-1} \overline{\mathfrak{bs}(A)}_i
\]
is an $F$-basis of $\mathrm {gr}(A)$. Of course one has $\rad(\mathrm {gr}(A))=\bigoplus_{i\geq 1} \rad^i(A)/ \rad^{i+1}(A)$.
Now, let $\overline{b_i}=b_i + \rad^{i+1}(A)\in \overline{\mathfrak{bs}(A)}_i$ and $\overline{b_j}=b_j + \rad^{j+1}(A)\in \overline{\mathfrak{bs}(A)}_j$, where $b_i,b_j\in \mathfrak{bs}(A)$.
If $b_ib_j\in \rad^{i+j+1}(A)$ then $\overline{b_i}\overline{b_j}=0$ in $\mathrm {gr}(A)$. Suppose then that $b_ib_j\notin \rad^{i+j+1}(A)$.
Since $\mathfrak{bs}(A)$ is a f.m.b. of $A$ one has $b_ib_j\in  \mathfrak{bs}(A)\cap \rad^{i+j}(A)$, so that  $\overline{b_i}\overline{b_j}\in \overline{\mathfrak{bs}(A)}_{i+j}$.
Therefore $\overline{\mathfrak{bs}(A)}$ is an f.m.b of $\mathrm {gr}(A)$, yielding the claim. \end{proof}

\smallskip
For every prime $p$ we will indicate by $F_p$ the field with $p$ elements. If $g$ and $h$ are two elements of a group then we will denote by $(g,h)$ their group commutator. We recall that a finite $p$-group $G$ is said to be \emph{powerful} if either $p=2$ and $G^\prime \subseteq G^4$ or $p>2$ and $G^\prime \subseteq G^p$. Here $G^\prime$ is the derived subgroup of $G$ and $G^k$ denotes the subgroup of $G$ generated by the elements $g^k$, $g\in G$.

\begin{corollary}\label{C:2}
Let $F G$  be the group algebra of a finite $p$-group $G$ over the  field $F$ of positive characteristic $p$. Denote by $\mathfrak{L}(G)$ the restricted Lie algebra associated with $G$.
Then the following statement hold:
\begin{itemize}
\item[(i)] if  $F G$ possesses an f.m.b. then so does $u(\mathfrak{L}(G)\otimes_{F_p}F)$;
\item[(ii)] if $p>2$ and $G$ is nilpotent of class $2$ then $F G$ does not have any f.m.b.
\end{itemize}
\end{corollary}
\begin{proof}
(i) We first recall the construction of $\mathfrak{L}(G)$ by means of the Zassenhaus-Jennings-Lazard series of $G$.
For every $n\in \mathbb{N}$ the $n$th dimension
subgroup of $G$ is defined by setting
\begin{equation}\label{E:2}
\mathfrak{D}_n(G) = G \cap(1 + \omega^n(FG)) =\prod_{ip^j\geq n}\gamma_i(G)^{p^j},
\end{equation}
where $\omega(FG)$ is the aumentation ideal of $F G$ and the $\gamma_i(G)$ are the terms of the descending central series of $G$.
Then the $F_p$-vector space
\[
{\mathfrak{L}}(G)=\bigoplus_{n\in \mathbb{N}} \mathfrak{D}_n(G)/\mathfrak{D}_{n+1}(G)
\]
has the structure of a restricted Lie algebra with respect to the Lie bracket and $p$-map defined by the following conditions:
\begin{equation}\label{E:3}
\begin{split}
[g\mathfrak{D}_{i+1}(G),h\mathfrak{D}_{i+1}(G)]&=(g,h)\mathfrak{D}_{i+j+1}(G),\\
 (g\mathfrak{D}_{i+1}(G))^p&=g^p\mathfrak{D}_{pi+1}(G).
\end{split}
\end{equation}
(For details we refer the reeder to Chapter VIII of \cite{Passi}.) Now, as $G$ is a $p$-group we clearly have $\rad(F G)=\omega(FG)$ and then, by a well-known theorem of Quillen in \cite{Quillen},  $\mathrm {gr}(F G)$ is isomorphic as an $F$-algebra to
the restricted enveloping algebra $u(\mathfrak{L}(G)\otimes_{F_p}F )$. Consequently, Theorem \ref{T:3} allows to conclude that $u(\mathfrak{L}\otimes_{F_p}F)$ has an f.m.b., as required.

(ii) If $G$ is nilpotent of class $2$ then it is clear that its associated restricted Lie algebra $\mathfrak{L}(G)$ is nilpotent of class at most $2$. Now, if  $\mathfrak{L}(G)$ is abelian, as $p>2$
and $\gamma_3(G)={1}$, from
(\ref{E:2}) and (\ref{E:3}) it follows that $\gamma_2(G) \subseteq  \mathfrak{D}_3(G)=G^{p}$. Therefore $G$ is powerful and so, in view of Theorem 1 of \cite{Bovdi_2}, the group algebra $F G$ cannot have an f.m.b.
On the other hand, if  $\mathfrak{L}(G)$ has nilpotence class $2$, then  by Theorem 3 of \cite{Bovdi_Grishkov_Siciliano} the restricted enveloping algebra $u(\mathfrak{L}(G)\otimes_{F_p}F )$ does not have any filtered multiplicative basis. Hence, from the part (1)  the claim follows at once. \end{proof}

The previous result gives a partial answer to question 5 in \cite{Oberwolfach}. Note also that a possible positive solution of Problem \ref{prob3} combined with Corollary \ref{C:2} would settle completely to question 5 in \cite{Oberwolfach}, as well.
Finally, it is worth remarking that, in general, the converse of Theorem \ref{T:3} is false. For instance, consider the following example:

{\bf Example}.  Let $F$ be a field of positive
characteristic $p$ containing an element $\alpha$
which is not a $p$-th root in $F$. Consider the abelian
restricted Lie algebra
\[
L_\alpha=F x+F y + F z
\]
with $x^{[p]}=\alpha z$,
$y^{[p]}=z$, and $z^{[p]}=0$. Note that $\rad(u(L))$ coincides with the augmentation ideal $\omega(L)$ of $u(L)$.
Consider the restricted Lie algebra
\[
\mathrm {gr}(L)=\bigoplus_{n\in \mathbb{N}} \mathfrak{D}_n(L)/\mathfrak{D}_{n+1}(L) \qquad (n\in \mathbb{N}).
\]
It is easy to see that  $\mathrm {gr}(L)$ is isomorphic to the direct sum of three cyclic restricted Lie algebras and so $u(\mathrm {gr}(L))$
has an f.m.b. (see \cite{Bovdi_Grishkov_Siciliano}, Theorem 1).
Moreover, by Theorem 2.2 of \cite{Usefi} one has $u(\mathrm {gr}(L))\cong \mathrm {gr}(u(L))$, hence $ \mathrm {gr}(u(L))$  has an f.m.b.

On the other hand, for what was showed in \cite{Bovdi_Grishkov_Siciliano} (see  the example on  page 607),  in this case $u(L)$ cannot have any filtered multiplicative basis.


\begin{thebibliography}{10}

\bibitem{Balogh_1}
Z.~Balogh.
\newblock On existing of filtered multiplicative bases in group algebras.
\newblock {\em Acta Math. Acad. Paedagog. Nyh\'azi. (N.S.)}, 20(1):11--30
  (electronic), 2004.

\bibitem{Balogh_2}
Z.~Balogh.
\newblock Further results on a filtered multiplicative basis of group algebras.
\newblock {\em Math. Commun.}, 12(2):229--238, 2007.

\bibitem{BGRS}
R.~Bautista, P.~Gabriel, A.~V. Ro{\u\i}ter, and L.~Salmer{\'o}n.
\newblock Representation-finite algebras and multiplicative bases.
\newblock {\em Invent. Math.}, 81(2):217--285, 1985.

\bibitem{Bovdi_A}
A.~A. Bovdi.
\newblock Group rings. ({R}ussian).
\newblock {\em Kiev.UMK VO}, page 155, 1988.

\bibitem{Bovdi_1}
V.~Bovdi.
\newblock On a filtered multiplicative basis of group algebras.
\newblock {\em Arch. Math. (Basel)}, 74(2):81--88, 2000.

\bibitem{Bovdi_2}
V.~Bovdi.
\newblock On a filtered multiplicative bases of group algebras. {II}.
\newblock {\em Algebr. Represent. Theory}, 6(3):353--368, 2003.

\bibitem{Bovdi_Grishkov_Siciliano}
V.~Bovdi, A.~Grishkov, and S.~Siciliano.
\newblock Filtered multiplicative bases of restricted enveloping algebras.
\newblock {\em Algebr. Represent. Theory}, 14(4):601--608, 2011.

\bibitem{Calderon}
A.~J. Calderon~Martin.
\newblock Associative algebras admitting a quasi-multiplicative basis.
\newblock {\em Algebr. Represent. Theory},
  (DOI:10.1007/s10468-014-9482-y):1--12, 2014.

\bibitem{Kupisch}
H.~Kupisch.
\newblock Symmetrische {A}lgebren mit endlich vielen unzerlegbaren
  {D}arstellungen. {I}.
\newblock {\em J. Reine Angew. Math.}, 219:1--25, 1965.

\bibitem{Landrock_Michler}
P.~Landrock and G.~O. Michler.
\newblock Block structure of the smallest {J}anko group.
\newblock {\em Math. Ann.}, 232(3):205--238, 1978.

\bibitem{Oberwolfach}
Mini-{W}orkshop: {A}rithmetik von {G}ruppenringen.
\newblock {\em Oberwolfach Rep.}, 4(4):3209--3239, 2007.
\newblock Abstracts from the mini-workshop held November 25--December 1, 2007,
  Organized by E.~Jespers, Z.~Marciniak, G.~Nebe and W.~Kimmerle, Oberwolfach
  Reports. Vol. 4, no. 4.

\bibitem{Passi}
I.~B.~S. Passi.
\newblock {\em Group rings and their augmentation ideals}, volume 715 of {\em
  Lecture Notes in Mathematics}.
\newblock Springer, Berlin, 1979.

\bibitem{Quillen}
D.~G. Quillen.
\newblock On the associated graded ring of a group ring.
\newblock {\em J. Algebra}, 10:411--418, 1968.

\bibitem{Riley_Shalev}
D.~M. Riley and A.~Shalev.
\newblock Restricted {L}ie algebras and their envelopes.
\newblock {\em Canad. J. Math.}, 47(1):146--164, 1995.

\bibitem{Roiter_Sergeichuk}
A.~V. Roiter and V.~V. Sergeichuk.
\newblock Existence of a multiplicative basis for a finitely spaced module over
  an aggregate.
\newblock {\em Ukrainian Math. J.}, 46(5):604--617 (1995), 1994.

\bibitem{Strade_Farnsteiner_book}
H.~Strade and R.~Farnsteiner.
\newblock {\em Modular {L}ie algebras and their representations}, volume 116 of
  {\em Monographs and Textbooks in Pure and Applied Mathematics}.
\newblock Marcel Dekker Inc., New York, 1988.

\bibitem{Usefi}
H.~Usefi.
\newblock Identifications in modular group algebras.
\newblock {\em J. Pure Appl. Algebra}, 212(10):2182--2189, 2008.

\bibitem{Zhu_Li}
H.~Zhu and F.~Li.
\newblock Weakly ordered multiplicative basis of an algebra related to quiver
  theory.
\newblock {\em J. Algebra}, 348:196--216, 2011.

\end{thebibliography}
\end{document}